\documentclass[12pt]{article}
\usepackage{graphicx}
\usepackage{amssymb,latexsym,amscd,amsmath,amsfonts,amsthm,pb-diagram,enumerate}
\usepackage[mathscr]{eucal}
\usepackage{floatflt}
\numberwithin{equation}{section}
\newtheorem{theorem}{Theorem}[section]
\newtheorem{proposition}[theorem]{Proposition}
\newtheorem{corollary}[theorem]{Corollary}
\newtheorem{lemma}[theorem]{Lemma}
\newtheorem{definition}[theorem]{Definition}

\newtheorem{remark}[theorem]{Remark}

\newtheorem{example}[theorem]{Example}
\begin{document}
\begin{center} {\Large \bf On the index of reducibility in Noetherian modules}
 \footnote {
\indent{\bf AMS Classification 2010}: Primary 13A15, 13C99; Secondary 13D45, 13H10.\\
This research is funded by Vietnam National Foundation for Science
and Technology Development (NAFOSTED) under grant numbers
101.04-2014.25 and 101.04-2014.15.}
\end{center}

\begin{center}
                {Nguyen Tu Cuong, Pham  Hung Quy and Hoang Le Truong}\\
\end{center}
\begin{abstract}
 Let $M$ be a finitely generated module over a Noetherian ring $R$ and $N$ a submodule. The index of reducibility ir$_M(N)$ is  the number of irreducible submodules that appear in an irredundant irreducible decomposition of $N$ (this number is well defined by a classical result of Emmy Noether). Then the main results of this paper are: (1) $\mathrm{ir}_M(N) = \sum_{\frak p \in \mathrm{Ass}_R(M/N)} \dim_{k(\frak p)} \mathrm{Soc}(M/N)_{\frak p} $; (2) For an irredundant primary decomposition of $N = Q_1 \cap \cdots \cap Q_n$, where $Q_i$ is $\frak p_i$-primary, then  $\mathrm{ir}_M(N) = \mathrm{ir}_M(Q_1) + \cdots + \mathrm{ir}_M(Q_n)$ if and only if $Q_i$ is a $\frak p_i$-maximal embedded component of $N$ for all embedded associated prime ideals $\frak p_i$ of $N$;
 (3) For an ideal $I$  of $R$ there exists a polynomial $\mathrm{Ir}_{M,I}(n)$ such that $\mathrm{Ir}_{M,I}(n)=\mathrm{ir}_M(I^nM)$ for  $n\gg 0$. Moreover, 
$\mathrm{bight}_M(I)-1\le \deg(\mathrm{Ir}_{M,I}(n))\le \ell_M(I)-1$; (4) If $(R, \frak m)$ is local,  $M$ is Cohen-Macaulay if and only if there exist an integer $l$ and a parameter ideal $\frak q$ of $M$ contained in $\frak m^l$ such that   $\mathrm{ir}_M({\frak q}M)=\dim_{R/\frak m}\mathrm{Soc}(H^d_{\frak m}(M))$, where $d=\dim M$. 
\end{abstract}

\section{Introduction}

One of the fundamental results in commutative algebra is the irreducible decomposition theorem \cite[Satz II and Satz IV]{N21} proved by Emmy Noether in 1921. In this paper she had showed that any ideal $I$ of a Noetherian ring $R$ can be expressed as a finite intersection of irreducible ideals, and the number of irreducible ideals in such an irredundant irreducible decomposition is independent of the choice of the decomposition. This number is then called the index of reducibility of $I$ and denoted by $\mathrm{ir}_R(I)$. Although irreducible ideals belong to basic objects of commutative algebra, there are not so much papers on the study of irreducible ideals and the index of reducibility. Maybe the first important paper on irreducible ideals after Noether's work is of W. Gr\"{o}bner \cite{G35} (1935). Since then there are interesting works on the index of reducibility of parameter ideals on local rings by D.G. Northcott \cite{No57} (1957), S. Endo and M. Narita \cite{EN64} (1964) or S. Goto and N. Suzuki \cite{GS84} (1984). Especially, W. Heinzer, L.J. Ratliff and K. Shah propounded in a series of papers \cite{HRS94}, \cite{HRS95-1}, \cite{HRS95-2}, \cite{HRS95-3} a theory of maximal embedded components which is useful for the study of irreducible ideals. It is clear that the concepts of irreducible ideals, the index of reducibility and maximal embedded components can be extended for finitely generated modules.  Then the purpose of this paper is to investigate the index of reducibility of submodules of a finitely generated $R$-module $M$  concerning its maximal embedded components as well as the behaviour of the function ir$_M(I^nM)$, where $I$ is an ideal of $R$,  and to present  applications of the index of reducibility for studying the structure of the module $M$. The paper is divided into 5 sections. Let $M$ be a finitely generated module over a Noetherian ring and $N$ a submodule of $M$. We present in the next section a formula to compute the index of reducibility $\mathrm{ir}_M(N)$ by using the socle dimension of the module $(M/N)_{\frak p}$ for all $\frak p \in \mathrm{Ass}_R(M/N)$ (see Lemma 2.3). This formula is a generalization of a well-known result which says that $\mathrm{ir}_M(N) = \dim_{R/\frak m} \mathrm{Soc}(M/N)$ provided $(R, \frak m)$ is a local ring and $\lambda_R(M/N) < \infty$. Section 3 is devoted to answer the following question: When is the index of reducibility of a submodule $N$  equal to the sum of the indices of reducibility of their primary components in a given irredundant primary decomposition of $N$? It turns out here that the notion of maximal embedded components of $N$ introduced by Heinzer, Ratliff and Shah is the key for answering this question (see Theorem \ref{T3.2}). In Section 4, we consider the index of reducibility $\mathrm{ir}_M(I^nM)$ of powers of an ideal $I$ as a function in $n$ and show that this function is in fact a polynomial for sufficiently large $n$. Moreover, we can prove that $\mathrm{bight}_M(I)-1$ is a lower bound and the $\ell_M(I)-1$ is an upper bound for the degree of this polynomial (see Theorem \ref{T4.1}), where $\mathrm{bight}_M(I)$ is the big height and 
 $\ell_M(I)$ is the analytic spread of $M$ with respect to the ideal $I$. However, the degree of this polynomial is still mysterious to us. We can only give examples to show that these bounds are optimal. In the last section, we involve in working out some applications of the index of reducibility. A classical result of Northcott  \cite{No57} says that the index of reducibility of a parameter ideal in a Cohen-Macaulay local ring is dependent only on the ring and not on the choice of parameter ideals. We will generalize Northcott's result in this section and get a characterization for the Cohen-Macaulayness of a Noetherian module in terms of the index of reducibility of parameter ideals (see Theorem \ref{T4.7}).

\section{Index of reducibility of submodules}
Throughout this paper $R$ is a Noetherian ring and $M$ is a finitely generated $R$-module. For an $R$-module $L$, $\lambda_R(L)$ denotes the length of $L$.

\begin{definition}\rm A submodule $N$ of $M$ is called an {\it irreducible submodule} if $N$ can not be written as an intersection of two properly larger submodules of $M$. The number of irreducible components of an irredundant irreducible decomposition of $N$, which is independent of the choice of the decomposition  by Noether \cite{N21}, is called the {\it index of reducibility} of $N$ and denoted by $\mathrm{ir}_M(N)$.
\end{definition}

\begin{remark}\rm We denoted by $\mathrm{Soc}(M)$ the sum of all simple submodules of $M$. $\mathrm{Soc}(M)$ is called the socle of $M$. If $R$ is a local ring with the unique maximal ideal $\frak m$ and $\frak k =  R/\frak m$ its residue field, then it is well-known that $\mathrm{Soc}(M) = 0:_M\frak m$ is a $\frak k$-vector space of finite dimension. Let $N$ be a submodule of $M$ with $\lambda_R(M/N) < \infty$. Then it is easy to check that $\mathrm{ir}_M(N) = \lambda_R((N:\frak m)/N) = \dim_{\frak k} \mathrm{Soc}(M/N).$
\end{remark}

The following lemma presents a formula for computing the index of reducibility $\mathrm{ir}_M(N)$  without the requirement that $R$ is local and $\lambda_R(M/N) < \infty$. It should be mentioned here that the first conclusion of the lemma would be known to experts. But, we cannot find its proof anywhere. So for the completeness, we give a short proof for it. Moreover, from this proof we obtain immediately a second conclusion which is useful for proofs of  further results in this paper. For a prime ideal $\frak p$, we use $k(\frak p)$ to denote the residue field $ R_{\frak p}/\frak pR_{\frak p}$ of the local ring $R_{\frak p}$.

\begin{lemma}\label{L2.3} Let $N$ be a submodule of $M$. Then
  $$\mathrm{ir}_M(N) = \sum_{\frak p \in \mathrm{Ass}_R(M/N)} \dim_{k(\frak p)} \mathrm{Soc}(M/N)_{\frak p} .$$
Moreover, for any $\frak p \in \mathrm{Ass}_R(M/N)$, there is a $\frak p$-primary submodule $N(\frak p)$ of $M$ with $\mathrm{ir}_M(N(\frak p)) = \dim_{k(\frak p)} \mathrm{Soc}(M/N)_{\frak p} $ such that
$$N = \bigcap_{\frak p \in \mathrm{Ass}_R(M/N)}N(\frak p)$$
is an irredundant primary decomposition of $N$.
\end{lemma}
\begin{proof}
  Passing to the quotient $M/N$ we may assume without any loss of generality that $N=0$. Let $\mathrm{Ass}_R(M) = \{ \frak p_1,..., \frak p_n\}$. We set $t_i = \dim_{k(\frak p_i)}\mathrm{Soc}(M_{\frak p_i})$ and $t = t_1 + \cdots + t_n$. Let $\mathcal{F} = \{\frak p_{11},..., \frak p_{1t_1}, \frak p_{21},..., \frak p_{2t_2}, ..., \frak p_{n1},..., \frak p_{nt_n}\}$ be a family of prime ideals of $R$ such that $\frak p_{i1} = \cdots = \frak p_{it_i} = \frak p_i$ for all $i = 1,..., n$. Denote $E(M)$ the injective envelop of $M$. Then we can write
  $$E(M) = \bigoplus_{i=1}^n E(R/\frak p_i)^{t_i} = \bigoplus_{\frak p_{ij} \in \mathcal{F}}E(R/\frak p_{ij}).$$
  Let $$ \pi_i: \oplus_{i=1}^n E(R/\frak p_i)^{t_i} \to E(R/\frak p_i)^{t_i} \ \ \mathrm{and} \ \  \pi_{ij}: \oplus_{\frak p_{ij} \in \mathcal{F}}E(R/\frak p_{ij}) \to E(R/\frak p_{ij})$$ be the canonical projections for all $i = 1,..., n$ and $j = 1,..., t_i$, and set $N(\frak p_i) = M \cap \ker \pi_i$, $N_{ij} = M \cap \ker \pi_{ij}$. Since $E(R/\frak p_{ij})$ are indecomposible, $N_{ij}$ are irreducible submodules of $M$. Then it is easy to check that $N(\frak p_i)$ is a $\frak p_i$-primary submodule of $M$ having an irreducible decomposition  $N(\frak p_i) = N_{i1} \cap \cdots \cap N_{it_i}$ for all $i=1,\ldots , n$. 
 Moreover, because of the minimality of $E(M)$ among injective modules containing $M$, the finite intersection $$0 = N_{11} \cap \cdots \cap N_{1t_1} \cap \cdots \cap N_{n1} \cap \cdots \cap N_{nt_n}$$ 
    is an irredundant irreducible decomposition of $0$. Therefore $0 = N(\frak p_1) \cap \cdots \cap N(\frak p_n)$ is an irredundant primary decomposition of $0$ with
  $\mathrm{ir}_M(N(\frak p_i)) = \dim_{k(\frak p_i)} \mathrm{Soc}(M/N)_{\frak p_i} $ and $\mathrm{ir}_M(0) = \sum_{\frak p \in \mathrm{Ass}(M)} \dim_{k(\frak p)} \mathrm{Soc}(M)_{\frak p}$
 as required.
 \end{proof}

\section{Index of reducibility of maximal embedded components}

Let $N$ be a submodule of $M$ and $\frak p \in \mathrm{Ass}_R(M/N)$. We use $\bigwedge_{\frak p}(N)$ to denote the set of all $\frak p$-primary submodules of $M$ which appear in an irredundant primary decomposition of $N$. We say that a $\frak p$-primary submodule $Q$ of $M$ is a $\frak p$-primary component of $N$ if $Q \in \bigwedge_{\frak p}(N)$, and $Q$ is said to be a {\it maximal embedded component} (or more precisely, $\frak p$-maximal embedded component) of $N$ if $Q$ is a maximal element in the set $\bigwedge_{\frak p}(N)$. It should be mentioned that the notion of maximal embedded components was first introduced for commutative rings  by Heinzer, Ratliff and Shah. They proved in the papers \cite{HRS94}, \cite{HRS95-1}, \cite{HRS95-2}, \cite{HRS95-3} many interesting properties of maximal embedded components as well as they showed that this notion is an important tool for  studying irreducible ideals.\\

We recall now a result of Y. Yao \cite{Y02} which is often used in the proof of the next theorem.
\begin{theorem}[Yao \cite{Y02}, Theorem 1.1] \label{T3.1} Let $N$ be a submodule of $M$, $\mathrm{Ass}_R(M/N) = \{\frak p_1,..., \frak p_n\}$ and $Q_i \in \bigwedge_{\frak p_i}(N)$, $i = 1,..., n$. Then $N = Q_1 \cap \cdots \cap Q_n$ is an irredundant primary decomposition of $N$.
\end{theorem}
The following theorem is the main result of this section.
\begin{theorem}\label{T3.2} Let $N$ be a submodule of $M$ and $\mathrm{Ass}_R(M/N) = \{\frak p_1,..., \frak p_n\}$. Let $N = Q_1 \cap \cdots \cap Q_n$ be an irredundant primary decomposition of $N$, where $Q_i$ is $\frak p_i$-primary for all $i= 1, \ldots , n$. Then $\mathrm{ir}_M(N) = \mathrm{ir}_M(Q_1) + \cdots + \mathrm{ir}_M(Q_n)$ if and only if $Q_i$ is a $\frak p_i$-maximal embedded component of $N$ for all embedded associated prime ideals $\frak p_i$ of $N$.
  \end{theorem}
  \begin{proof} As in the proof of Lemma \ref{L2.3}, we may assume that $N = 0$.\\
  {\it Sufficient condition}: Let $0 = Q_1 \cap \cdots \cap Q_n$ be an irredundant primary decomposition of the zero submodule $0$, where $Q_i$ is maximal in $\bigwedge_{\frak p_i}(0)$, $i = 1,..., n$. Setting $\mathrm{ir}_M(Q_i) = t_i$, and let $Q_i = Q_{i1} \cap \cdots \cap Q_{it_i}$ be an irredundant irreducible decomposition of $Q_i$. Suppose that
  $$t_1 + \cdots + t_n = \mathrm{ir}_M(Q_1) + \cdots + \mathrm{ir}_M(Q_n) > \mathrm{ir}_M(0).$$
  Then there exist an $i \in \{1,..., n \}$ and a $j \in \{1,..., t_i\}$ such that
  $$Q_1 \cap \cdots \cap Q_{i-1} \cap Q_i' \cap Q_{i+1} \cap \cdots \cap Q_n \subseteq Q_{ij},$$
  where $Q_i' = Q_{i1} \cap \cdots \cap Q_{i(j-1)} \cap Q_{i(j+1)} \cap \cdots \cap Q_{it_i} \supsetneqq Q_i$. Therefore
  $$Q_i' \bigcap (\cap_{k \neq i} Q_k) = Q_i \bigcap (\cap_{k \neq i} Q_k) = 0$$
  is also an irredundant primary decomposition of $0$. Hence $Q_i' \in \bigwedge_{\frak p_i}(0)$ which contradicts the maximality of $Q_i$ in $\bigwedge_{\frak p_i}(0)$. Thus $\mathrm{ir}_R(0) = \mathrm{ir}_R(Q_1) + \cdots + \mathrm{ir}_R(Q_n)$ as required.\\
  {\it Necessary condition}: Assume that $0 = Q_1 \cap \cdots \cap Q_n$ is an irredundant primary decomposition of $0$ such that $\mathrm{ir}_M(0) = \mathrm{ir}_M(Q_1) + \cdots + \mathrm{ir}_M(Q_n)$. We have to prove that $Q_i$ are maximal in $\bigwedge_{\frak p_i}(0)$ for all $i = 1,..., n$. Indeed, let $N_1=N(\frak p_1),..., N_n=N(\frak p_n)$ be primary submodules of $M$ as in Lemma \ref{L2.3}, that is $N_i \in \bigwedge_{\frak p_i}(0)$, $0 = N_1 \cap \cdots \cap N_n$ and $\mathrm{ir}_M (0)  = \sum_{i=1}^n \mathrm{ir}_M(N_i) = \sum_{i=1}^n \dim_{k(\frak p_i)} \mathrm{Soc}(M_{\frak p_i})$. Then by Theorem \ref{T3.1} we see for any $0 \leq i \leq n$ that
  $$0 = N_1 \cap \cdots \cap N_{i-1} \cap Q_i \cap N_{i+1} \cap \cdots \cap N_n = N_1 \cap \cdots \cap N_n$$
are two irredundant primary decompositions of $0$. Therefore
$$\mathrm{ir}_M(Q_i) + \sum_{j \neq i} \mathrm{ir}_M(N_j) \geq \mathrm{ir}_M(0) = \sum_{j =1}^n \mathrm{ir}_M(N_j),$$
and so $\mathrm{ir}_M(Q_i) \geq \mathrm{ir}_M(N_i) = \dim_{k(\frak p_i)} \mathrm{Soc}(M_{\frak p_i})$ by Lemma \ref{L2.3}.\\
Similarly, it follows from the two irredundant primary decompositions
$$0 = Q_1 \cap \cdots \cap Q_{i-1} \cap N_i \cap Q_{i+1} \cap \cdots \cap Q_n = Q_1 \cap \cdots \cap Q_n$$
and the hypothesis that $\mathrm{ir}_M(N_i) \geq \mathrm{ir}_M(Q_i)$. Thus we get
$$\mathrm{ir}_M(Q_i) = \mathrm{ir}_M(N_i) = \dim_{k(\frak p_i)} \mathrm{Soc}(M_{\frak p_i})$$
for all $i = 1, ..., n$. Now, let $Q_i'$ be a maximal element of  $\bigwedge_{\frak p_i}(0)$ and $Q_i \subseteq Q_i'$. It remains to prove that $Q_i = Q_i'$. By localization at $\frak p_i$, we may assume that $R$ is a local ring with the unique maximal ideal $\frak m = \frak p_i$. Then, since $Q_i$ is an $\frak m$-primary submodule and by the equality above we have
$$\lambda_R((Q_i:\frak m)/Q_i) = \mathrm{ir}_M(Q_i) = \dim_{\frak k}\mathrm{Soc}(M) = \lambda_R(0:_M \frak m) = \lambda_R \big( (Q_i + 0:_M \frak m)/Q_i\big).$$
It follows that $Q_i: \frak m = Q_i + 0:_M \frak m$. If $Q_i \subsetneqq Q_i'$, there is an element $x \in Q_i' \setminus Q_i$. Then we can find a positive integer $l$ such that $\frak m^l x \subseteq Q_i$ but $\frak m^{l-1} x \nsubseteq Q_i$. Choose $y \in \frak m^{l-1} x \setminus Q_i$. We see that
$$y \in Q_i' \cap (Q_i : \frak m) = Q_i' \cap (Q_i + 0:_M \frak m) = Q_i + (Q_i' \cap 0:_M \frak m).$$
Since $0:_M \frak m \subseteq \cap_{j \neq i} Q_j$ and $Q_i' \cap (\cap_{j \neq i} Q_j) = 0$ by Theorem \ref{T3.1}, $Q_i' \cap (0:_M \frak m) = 0$. Therefore $y \in Q_i$ which is a contradiction with the choice of $y$. Thus $Q_i = Q_i'$ and the proof is complete.
  \end{proof}

The following characterization of maximal embedded components of $N$ in terms of the index of reducibility follows immediately from the proof of Theorem \ref{T3.2}.
\begin{corollary}\label{T3.3}
Let $N$ be a submodule of $M$ and $\frak p$ an embedded associated prime ideal of $N$. Then an element $Q \in \bigwedge_{\frak p}(N)$ is a maximal embedded component of $N$ if and only if $\mathrm{ir}_M(Q) = \dim_{k(\frak p)} \mathrm{Soc}(M/N)_{\frak p}$.
\end{corollary}

As consequences of Theorem \ref{T3.2}, we can obtain again several results on maximal embedded components proved by Heinzer, Ratliff and Shah. The following corollary is one of that results stated for modules. For a submodule $L$ of $M$ and $\frak p$ a prime ideal, we denote by $\mathrm{IC}_{\frak p}(L)$ the set of all irreducible $\frak p$-primary submodules of $M$ that appear in an irredundant irreducible decomposition of $L$, and denote by $\mathrm{ir}_{\frak p}(L)$ the number of irreducible $\frak p$-primary components in an irredundant irreducible decomposition of $L$ (this number is well defined by Noether \cite[Satz VII]{N21}).

\begin{corollary}[see \cite{HRS95-3}, Theorems 2.3 and 2.7] Let $N$ be a submodule of $M$ and $\frak p$ an embedded associated prime ideal of $N$. Then
\begin{enumerate}[{(i)}]\rm
\item {\it  $\mathrm{ir}_{\frak p}(N) = \mathrm{ir}_{\frak p}(Q) = \dim_{k(\frak p)} \mathrm{Soc}(M/N)_{\frak p}$ for any $\frak p$-maximal embedded component $Q$ of $N$.}
\item {\it $\mathrm{IC}_{\frak p}(N) = \bigcup_Q \mathrm{IC}_{\frak p}(Q)$, where the submodule $Q$ in the union runs over all $\frak p$-maximal embedded components of $N$.}
\end{enumerate}
\end{corollary}

\begin{proof}
(i) follows immediately from the proof of Theorem \ref{T3.2} and Corollary \ref{T3.3}.\\
(ii) Let $Q_1 \in \mathrm{IC}_{\frak p}(N)$ and $t_1 = \dim_{k(\frak p)} \mathrm{Soc}(M/N)_{\frak p}$. By the hypothesis and (i) there exists an irredundant irreducible decomposition $N= Q_{11}\cap  \ldots  \cap Q_{1 t_1} \cap Q_2 \cap \ldots \cap Q_l$ such that $Q_{11} = Q_1 , \  Q_{12},  \ldots , Q_{1 t_1}$ are all $\frak p$-primary submodules in this decomposition. Therefore $Q=Q_{11}\cap  \ldots  \cap Q_{1 t_1}$ is a maximal embedded component of $N$ by Corollary \ref{T3.3}, and so $Q_1\in \mathrm{IC}_{\frak p}(Q)$. The converse inclusion can be easily proved by applying Theorems \ref{T3.1} and \ref{T3.2}.
\end{proof}

\section{Index of reducibility of powers of an ideal}
Let $I$ be an ideal of $R$. It is well known by \cite{B79} that the $\mathrm{Ass}_R(M/I^nM)$ is stable for sufficiently large  $n$ ($n\gg 0$ for short). We will denote this stable set by $\text{A}_M(I)$. The big height, $\mathrm{bight}_M(I)$, of $I$ on $M$ is defined by
$$\mathrm{bight}_M(I)=\max\{\dim_{R_{\frak p}} M_{\frak p}\mid \text{ for all minimal prime ideals } {\frak p}\in \mathrm{Ass}_R(M/IM)\}.$$
Let $G(I)=\bigoplus\limits_{n\geq 0} I^n/I^{n+1}$ be the associated graded ring of $R$ with respect to $I$ and $G_M(I)=\bigoplus\limits_{n\geq 0} I^nM/I^{n+1}M$ the associated graded $G(I)$-module of $M$ with respect to $I$. If $R$ is a local ring with the unique maximal ideal $\frak m$, then the analytic spread $\ell_M(I)$ of $I$ on $M$ is defined by
$$\ell_M(I)=\dim_{G(I)}(G_M(I)/{\frak m} G_M(I)).$$
If $R$ is not local, the analytic spread $\ell_M(I)$ is also defined by
$$
\begin{aligned}
\ell_M(I)=\max\{ \ell_{M_{\frak m}}(IR_{\frak m})\mid {\frak m} &\text{ is a maximal ideal and }
\\&
\text{ there is a prime ideal }{\frak p} \in A_M(I) \text{ such that } {\frak p}\subseteq {\frak m}\}.
\end{aligned}$$
We use $\ell(I)$ to denote the analytic spread of the ideal $I$ on $R$.
The following theorem is the main result of this section.

\begin{theorem} \label{T4.1} Let $I$ be an ideal of $R$. Then there exists a polynomial $\mathrm{Ir}_{M,I}(n)$ with rational coefficients such that $\mathrm{Ir}_{M,I}(n)=\mathrm{ir}_M(I^nM)$ for sufficiently large  $n$. Moreover, we have
$$\mathrm{bight}_M(I)-1\le \deg(\mathrm{Ir}_{M,I}(n))\le \ell_M(I)-1.$$
\end{theorem}

To prove Theorem \ref{T4.1}, we need the following lemma.

\begin{lemma}\label{L4.2} Suppose that $R$ is a local ring with the unique maximal ideal $\frak m$ and $I$ an ideal of $R$. Then
\begin{enumerate}[{(i)}]\rm
    \item {\it $\dim_{\frak k}\mathrm{Soc}(M/I^nM)=\lambda_R(I^nM:{\frak m}/I^nM)$ is a polynomial of degree $\le \ell_M(I)-1$ for  $n\gg 0$.}
    \item {\it Assume  that $I$ is an ${\frak m}$-primary ideal. Then $\mathrm{ir}_M(I^nM)=\lambda_R(I^nM:{\frak m}/I^nM)$ is a polynomial of degree $\dim_RM-1$ for  $n\gg 0$.}
  \end{enumerate}
  \end{lemma}

  \begin{proof}

(i) Consider the homogeneous submodule $0:_{G_M(I)}{\frak m}G(I)$. Then 
$$\lambda_R(0:_{G_M(I)}{\frak m}G(I))_n=\lambda_R(((I^{n+1}M:{\frak m})\cap I^nM)/I^{n+1}M)$$ is a polynomial for $n\gg 0$. Using a result proved by P. Schenzel \cite[Proposition 2.1]{S98}, proved  for rings but  easily extendible to  modules, we find a positive integer $l$ such that for all $n\ge l$, $0:_M{\frak m}\cap I^nM=0$ and
$$I^{n+1}M:{\frak m}= I^{n+1-l}(I^l M:{\frak m})+ 0:_M{\frak m}.$$
Therefore 
$$
\begin{aligned}
(I^{n+1}M:{\frak m})\cap I^nM&=I^{n+1-l}(I^l M:{\frak m})+0:_M{\frak m}\cap I^nM\\
&=I^{n+1-l}(I^l M:{\frak m}).
\end{aligned}
$$
Hence, $\lambda_R(I^{n+1-l}(I^l M:{\frak m})/I^{n+1}M)=\lambda_R(((I^{n+1}M:{\frak m})\cap I^nM)/I^{n+1}M)$ is a polynomial for $n\gg 0$. It follows that
$$\dim_{\frak k}\mathrm{Soc}(M/I^nM) =\lambda_R((I^n M:{\frak m})/I^{n}M)=\lambda_R(I^{n-l}(I^l M:{\frak m})/I^{n}M)+\lambda_R(0:_M{\frak m})$$
is a polynomial for $n\gg 0$, and the degree of this polynomial is just equal to
$$\dim_{G(I)}(0:_{G_M(I)}{\frak m}G(I))-1\le \dim_{G(I)}({G_M(I)}/{\frak m}G_M(I))-1=\ell_M(I)-1.$$
(ii) The second statement follows from the first one and the fact that
$$
\begin{aligned}
\lambda_R(I^nM/I^{n+1}M) &=\lambda_R(\mathrm{Hom}_R(R/I,I^nM/I^{n+1}M))\\
 &\le \lambda_R(R/I)\lambda_R(\mathrm{Hom}_R(R/{\frak m},I^nM/I^{n+1}M)) \le \lambda_R(R/I)\mathrm{ir}_M(I^{n+1}M).
\end{aligned}
$$
\end{proof}

We are now able to prove Theorem \ref{T4.1}.

\begin{proof}[Proof of Theorem \ref{T4.1}] Let $\text{A}_M(I)$ denote  the stable set  $\mathrm{Ass}_R(M/I^nM)$ for $n\gg 0$.   Then, by Lemma \ref{L2.3} we get that
$$\mathrm{ir}_M(I^nM)=\sum\limits_{\frak p\in A_M(I)}\dim_{k({\frak p})}\mathrm{Soc}(M/I^nM)_{\frak p}$$ for all $n\gg 0$.
From Lemma \ref{L4.2}, (i), $\dim_{k({\frak p})}\mathrm{Soc}(M/I^nM)_{\frak p}$ is a polynomial of degree $\le \ell_{M_{\frak p}}(IR_{\frak p})-1$ for  $n\gg0$. Therefore there exists a polynomial $\mathrm{Ir}_{M,I}(n)$ of  such that $\mathrm{Ir}_{M,I}(n)=\mathrm{ir}_M(I^nM)$ for $n\gg 0$ and 
$$\mathrm{deg}(\mathrm{Ir}_{M,I}(n))\le \max \{\ell_{M_{\frak p}}(IR_{\frak p})-1\mid \frak p \in A_M(I)\}\le \ell_M(I)-1.$$
Let $\mathrm{Min}(M/IM)=\{{\frak p_1},\ldots,{\frak p_m}\}$ be the set of all minimal associated prime ideals of $IM$. It is clear that ${\frak p_i}$ is also minimal in $A_M(I)$. Hence $\Lambda_{\frak p_i}(I^nM)$ has only one element, says $Q_{in}$. It is easy to check that 
$$\mathrm{ir}_M(Q_{in})={\rm ir}_{M_{\frak p_i}}(Q_{in})_{\frak p_i}=\mathrm{ir}_{M_{\frak p_i}}(I^nM_{\frak p_i})$$ for $i=1, \ldots , m$.  This implies by Theorem \ref{T3.2} that
$\mathrm{ir}_M(I^nM)\ge \sum\limits_{i=1}^m \mathrm{ir}_{M_{\frak p_i}}(I^nM_{\frak p_i})$. It follows from Lemma \ref{L4.2}, (ii) for  $n\gg 0$ that
$$\mathrm{deg}(\mathrm{Ir}_{M,I}(n))\ge \max\{\dim_{R_{\frak p_i} }M_{\frak p_i}-1\mid i=1, \ldots ,m\}=\mathrm{bight}_M(I)-1.$$
\end{proof}

 The following corollaries are immediate consequences of Theorem \ref{T4.1}.
 An ideal $I$ of a local ring $R$ is called an equimultiple ideal if  $\ell(I)=\mathrm{ht}(I)$, and therefore $\mathrm{bight}_R(I)=\mathrm{ht}(I)$.

\begin{corollary} \label{C4.3} Let $I$ be an ideal of $R$ satisfying $\ell_M(I)= \mathrm{bight}_M(I)$. Then $$\deg(\mathrm{Ir}_{M,I}(n))=\ell_M(I)-1.$$
\end{corollary}

\begin{corollary} \label{C4.4} Let $I$ be an equimultiple ideal of a local ring $R$ with the unique maximal ideal $\frak m$. Then  $$\mathrm{deg}(\mathrm{Ir}_{R,I}(n))=\mathrm{ht}(I)-1$$. 
\end{corollary}

Excepting the corollaries above, the authors of the paper do not know how to compute exactly the degree of the polynomial of index of reducibility $\mathrm{Ir}_{M,I}(n)$. Therefore it is maybe interesting to find a formula for this degree in terms of known invariants associated to $I$ and $M$. Below we give examples to show that although these bounds are sharp, neither  $\mathrm{bight}_M(I)-1$ nor  $\ell_M(I)-1$ equal to $\mathrm{deg}(\mathrm{Ir}_{M,I}(n))$ in general.

\begin{example} \label{E4.5}{\rm
(1) Let $R=K[X,Y]$ be the polynomial ring of two variables $X$, $Y$ over a field $K$ and $I=(X^2,XY)=X(X,Y)$ an ideal of $R$. Then we have
$$\mathrm{bight}_R(I)=\mathrm{ht}(I)=1, \text{  }\ell(I)=2,$$ and by Lemma \ref{L2.3}
$$\mathrm{ir}_R(I^n)=\mathrm{ir}_R(X^n(X,Y)^n)=\mathrm{ir}_R((X,Y)^n) +1=n+1.$$
Therefore 
$$\mathrm{bight}_R(I)-1=0 <1= \mathrm{deg}(\mathrm{Ir}_{R,I}(n))=\ell (I)-1.$$

(2) Let $T=K[X_1,X_2,X_3,X_4,X_5,X_6]$ be the polynomial ring in six variables over a field $K$ and $R=T_{(X_1,\ldots,X_6)}$ the localization of $T$ at the homogeneous maximal ideal $(X_1,\ldots,X_6)$. Consider the monomial ideal
$$
\begin{aligned}
I&=(X_1X_2,X_2X_3,X_3X_4,X_4X_5,X_5X_6,X_6X_1)= (X_1,X_3,X_5)\cap (X_2,X_4,X_6)\cap \\
&\cap (X_1,X_2,X_4,X_5) \cap (X_2,X_3,X_5,X_6)\cap (X_3,X_4,X_6,X_1).
\end{aligned}
$$
Since the associated graph to this monomial ideal is a bipartite graph, it follows from \cite[Theorem 5.9]{SVV94} that
$\mathrm{Ass}(R/I^n)=\mathrm{Ass}(R/I)=\mathrm{Min}(R/I)$ for all $n\geq 1$. Therefore $\mathrm{deg}(\mathrm{Ir}_{R,I}(n))= \mathrm{bight}(I)-1= 3$ by Theorem \ref{T3.2} and Lemma \ref{L4.2} (ii). On the other hand, by \cite[Exercise 8.21]{HS06} $\ell(I)=5$, so
$$\mathrm{deg}(\mathrm{Ir}_{R,I}(n))=3< 4=\ell(I)-1.$$
}
\end{example}

Let $I$ be an ideal of $R$ and $n$ a positive integer. The $n$th symbolic power $I^{(n)}$ of $I$ is defined by
$$I^{(n)} = \bigcap_{\frak p \in \mathrm{Min}(I)}(I^nR_{\frak p}\cap R),$$ where ${\rm Min} (I)$ is the set of all minimal associated prime ideals in Ass$(R/I)$. Contrary to the function ${\rm ir}(I^n)$, the behaviour of the function ${\rm ir}(I^{(n)})$ seems to be better.
\begin{proposition} \label{C4.2}
Let $I$ be an ideal of $R$. Then there exists a polynomial $p_I(n)$ with rational coefficients that such
$p_I(n)={\rm ir}_R (I^{(n)}) $ for  sufficiently large $n$ and $${\rm deg}(p_I(n))={\rm bight}(I)-1.$$
\end{proposition} \label{C4.2}
\begin{proof} It should be mentioned that Ass$(R/I^{(n)})={\rm Min} (I)$ for all positive integer $n$. Thus, by virtue of Theorem \ref{T3.2}, we can show as in the proof of Theorem \ref{T4.1} that 
$${\rm ir} _R(I^{(n)})=\sum\limits_{\frak p\in {\rm Min}(I)}\mathrm{ir}_{R_{\frak p}}(I^nR_{\frak p})$$ for all $n$. So the proposition follows from Lemma \ref{L4.2}, (ii).
\end{proof}

\section{Index of reducibility in Cohen-Macaulay modules}
In this section, we assume in addition that $R$ is a local ring with the unique maximal ideal $\frak m$, and $\frak k = R/\frak m$ is the residue field. Let $\frak q=(x_1,\ldots, x_d)$ be a parameter ideal of $M$ ($d=\dim M$). Let $H^i({\frak q}, M)$ be the $i$-th Koszul cohomology module of $M$ with respect to $\frak q$ and $H^i_{\frak m}(M)$ the $i$-th local cohomology module of $M$ with respect to the maximal ideal $\frak m$. In order to state the next theorem, we need the following result of Goto and Sakurai \cite[Lemma 3.12]{GS03}.

\begin{lemma}\label{L4.6}
There exists a positive integer $l$ such that for all parameter ideals $\frak q$ of $M$ contained in $\frak m^l$, the canonical homomorphisms on socles
$$\mathrm{Soc}(H^i({\frak q}, M))\to \mathrm{Soc}(H^i_{\frak m}(M))$$
are surjective for all $i$.
 \end{lemma}

\begin{theorem} \label{T4.7} Let $M$ be a finitely generated $R$-module of $\dim M=d$. Then the following conditions are equivalent:
\begin{enumerate}[{(i)}]\rm
    \item {\it $M$ is a Cohen-Macaulay module.}
    \item {\it $\mathrm{ir}_M({\frak q}^{n+1}M)=\dim_{\frak k}\mathrm{Soc}(H^d_{\frak m}(M)) \binom{n+d-1}{d-1}$ for all parameter ideals $\frak q$ of $M$ and all $n \geq 0$.}
    \item {\it  $\mathrm{ir}_M({\frak q}M)=\dim_{\frak k} \mathrm{Soc}(H^d_{\frak m}(M))$ for all parameter ideals $\frak q$ of $M$.}
    \item {\it There exists a parameter ideal $\frak q$ of $M$ contained in $\frak m^l$, where $l$ is a positive integer as in Lemma \ref{L4.6}, such that $\mathrm{ir}_M({\frak q}M)=\dim_{\frak k}\mathrm{Soc}(H^d_{\frak m}(M))$.}
  \end{enumerate}
\end{theorem}
\begin{proof}
(i) $\Rightarrow$ (ii) Let $\frak q$ be a parameter ideal of $M$. Since $M$ is Cohen-Macaulay, we have a natural isomorphism of graded modules

$$G_M(\frak q)=\bigoplus\limits_{n\ge 0}\frak q^nM/\frak q^{n+1}M\to M/\frak q M[T_1,\ldots,T_d],$$ where $T_1,\ldots,T_d$ are indeterminates.
This deduces $R$-isomomorphisms on graded parts
$$\frak q^nM/\frak q^{n+1}M\to\big( M/\frak q M[T_1,\ldots,T_d]\big)_n\cong M/\frak q M^{\binom{n+d-1}{d-1}}$$ for all $n\geq 0$.
On the other hand, since $\frak q$ is a parameter ideal of a Cohen-Macaulay module, $\frak q^{n+1}M:\frak m\subseteq \frak q^{n+1}M:\frak q=\frak q^{n}M$. It follows that
$$
\begin{aligned}
\mathrm{ir}_M({\frak q}^{n+1}M)&=\lambda_R(\frak q^{n+1}M:\frak m/\frak q^{n+1}M)=\lambda_R(0:_{\frak q^{n}M/\frak q^{n+1}M}\frak m)\\
&=\lambda_R(0:_{M/\frak q M}\frak m) \binom{n+d-1}{d-1}=\dim_{\frak k}(\mathrm{Soc}(M/\frak q M))\binom{n+d-1}{d-1}.
\end{aligned}
$$
So the conclusion is proved, if we show that $\dim_{\frak k} \mathrm{Soc}(M/\frak q M)=\dim_{\frak k} \mathrm{Soc}(H^d_{\frak m}(M))$. Indeed, let $\frak q=(x_1,\ldots, x_d)$ and $\overline{M}= M/x_1M$. Then, it is easy to show by induction on $d$ that
$$
\begin{aligned}
\dim_{\frak k} \mathrm{Soc}(M/\frak q M)&=\dim_{\frak k} \mathrm{Soc}(\overline{M}/\frak q \overline{M})\\
&=\dim_{\frak k} \mathrm{Soc}(H^{d-1}_{\frak m}(\overline{M})) = \dim_{\frak k} \mathrm{Soc}(H^d_{\frak m}(M)).
\end{aligned}
$$

(ii) $\Rightarrow$ (iii) and (iii) $\Rightarrow$ (iv) are trivial.

(iv) $\Rightarrow$ (i) Let $\frak q =(x_1, \ldots , x_d)$ be a parameter ideal of $M$ such that $\frak q\subseteq \frak m^l$, where $l$ is a positive integer as in Lemma 5.1 such that the canonical homomorphism on socles
$$\mathrm{Soc}(M/\frak q M)=\mathrm{Soc}(H^d({\frak q}, M))\to \mathrm{Soc}(H^d_{\frak m}(M))$$
is surjective. Consider the submodule  $(\underline{x})_M^{\mathrm{lim}}=\bigcup\limits_{t\ge 0}(x_1^{t+1},\ldots,x_d^{t+1}):(x_1\ldots x_d)^{t}$ of $M$. This submodule is called the limit closure of the sequence $x_1,\ldots,x_d$. Then 
$(\underline{x})_M^{\mathrm{lim}}/\frak q M$ is just the kernel of the canonical homomorphism $M/\frak q M\to H^d_{\frak m}(M)$ (see \cite{CHL99}, \cite{CQ10}). Moreover, it was proved in \cite[Corollary 2.4]{CHL99} that the module $M$ is Cohen-Macaulay if and only if $(\underline{x})_M^{\mathrm{lim}}=\frak q M$. Now we assume that $\mathrm{ir}_M({\frak q}M)=\dim_{\frak k} \mathrm{Soc}(H^d_{\frak m}(M))$, therefore $\dim_{\frak k} \mathrm{Soc}(H^d_{\frak m}(M)) =\dim_{\frak k} \mathrm{Soc}(M/\frak q M)$. Then it follows from the exact sequence
$$0\to (\underline{x})_M^{\mathrm{lim}}/\frak q M \to M/\frak q M\to H^d_{\frak m}(M) $$
 and the choice of $l$ that the sequence
$$0\to \mathrm{Soc}((\underline{x})_M^{\mathrm{lim}}/\frak q M) \to \mathrm{Soc}( M/\frak q M)\to\mathrm{Soc}( H^d_{\frak m}(M))\to 0 $$
 is a short exact sequence. Hence $\dim_{\frak k} \mathrm{Soc}((\underline{x})_M^{\mathrm{lim}}/\frak q M)=0 $ by the hypothesis. So $(\underline{x})_M^{\mathrm{lim}}=\frak q M$, and therefore $M$ is a Cohen-Macaulay module.
  \end{proof}
  \rm It should be mentioned here that the proof of implication (iv) $\Rightarrow$ (i) of Theorem \ref{T4.7} is essentially following the proof  of \cite [Theorem 2.7] {MRS08}.  It is well-known that a Noetherian local ring $R$ with $\dim R =d$ is Gorenstein if and only if $R$ is Cohen-Macaulay with the Cohen-Macaulay 
type r$(R)=\dim_{\frak k}\mathrm{Ext}^d(\frak k,M))= 1$. Therefore the following result, which is the main result of \cite[Theorem]{MRS08}, is an immediate consequence of Theorem \ref{T4.7}.
 
 \begin{corollary}\label{5.3}
 Let $(R, \frak m)$ be a Noetherian local ring of dimension $d$. Then $R$ is Gorenstein if and only if there exists an irreducible parameter ideal $\frak q$ contained in $\frak m^l$, where $l$ is a positive integer as in Lemma \ref{L4.6}. Moreover, if $R$ is Gorenstein, then for any parameter ideal $\frak q$ it holds ${\rm ir}_R(\frak q^{n+1}) = \binom{n+d-1}{d-1}$ for all $n\geq 0$. 
 \end{corollary}
 \begin{proof}
Let $\frak q=(x_1,\ldots , x_d)$ be an irreducible parameter ideal contained in $\frak m^l$ such that the map
 $$\mathrm{Soc}( M/\frak q M)\to\mathrm{Soc}( H^d_{\frak m}(M))$$ is surjective.
 Since $\dim_{\frak k}\mathrm{Soc}( H^d_{\frak m}(M))\not= 0$ and $\dim_{\frak k} \mathrm{Soc}(M/\frak q M)=1$ by the hypothesis, 
 $\dim_{\frak k} \mathrm{Soc}( H^d_{\frak m}(M))=1. $ This imples by Theorem \ref{T4.7} that  $M$ is a Cohen-Macaulay module with 
 $${\rm r}(R)=\dim_{\frak k}\mathrm{Ext}^d(\frak k,M)=\dim_{\frak k} \mathrm{Soc}(M/\frak q M)=  1,$$ and so $R$ is Gorenstein. The last conclusion follows from Theorem \ref{T4.7}.
 \end{proof}
 \begin{remark} \rm  Recently, it was shown by many works that the index of reducibility of parameter ideals can be used to deduce a lot of information on the structure of some classes of modules such as Buchsbaum modules \cite{GS03}, generalized Cohen-Macaulay modules \cite{CT08}, 
\cite{Q12} and sequentially Cohen-Macaulay modules \cite{T13}.
 It follows from Theorem \ref{T4.7} that $M$ is a Cohen-Macaulay module if and only if there exists a positive integer $l$ such that
    $\mathrm{ir}_M({\frak q}M)=\dim_{\frak k}\mathrm{Soc}(H^d_{\frak m}(M))$ for all parameter ideals $\frak q$ of $M$ contained in $\frak m^l$. The necessary condition of this result can be extended for a large class of modules called generalized Cohen-Macaulay modules. An $R$-module $M$ of dimension $d$ is said to be a generalized Cohen-Macaulay module (see \cite{CST78}) if $H^i_{\frak m}(M)$ is of finite length for all $i=0,\ldots , d-1$. We proved in \cite [Theorem 1.1]{CT08} (see also \cite [Corollary 4.4]{CQ11}) that if $M$ is a generalized Cohen-Macaulay module, then there exists an integer $l$ such that 
  $${\rm ir}_M(\frak qM) = \sum_{i=0}^d \binom{d}{i}\dim_{\frak k}\mathrm{Soc}(H^i_{\frak m}(M)).$$
  for all parameter ideals $\frak q \subseteq \frak m^l$. Therefore, we close this paper with the following two open questions, which are suggested during the work in this paper, on the characterization of the Cohen-Macaulayness and of the generalized Cohen-Macaulayness in terms of the index of reducibility of parameter ideals as follows.
  \vskip0.1cm
  \noindent
  {\bf Open questions 5.5.} Let $M$ be a finitely generated module of dimension $d$ over a local ring $R$. Then our questions are \\
  1. Is $M$  a Cohen-Macaulay module if and only if there exists a parameter ideal $\frak q$ of $M$ such that $$\mathrm{ir}_M({\frak qM}^{n+1}M)=\dim_{\frak k}\mathrm{Soc}(H^d_{\frak m}(M)) \binom{n+d-1}{d-1}$$ for all $n\geq 0$?
 
  \noindent
  2. Is $M$ a generalized Cohen-Macaulay module if and only if  there exists  a positive integer $l$ such that  $${\rm ir}_M(\frak qM)= \sum_{i=0}^d \binom{d}{i}\dim_{\frak k}\mathrm{Soc}(H^i_{\frak m}(M))$$ for all parameter ideals $\frak q \subseteq \frak m^l$?
 \end{remark}
 
\bigskip
\noindent{\bf Acknowledgments.} The authors would like to thank the anonymous referee for helpful comments on the earlier version. This paper was finished  during the  authors' visit at the Vietnam
Institute for Advanced Study in Mathematics (VIASM).
They would like to thank VIASM for their support and hospitality.

Institute of Mathematics, Vietnam Academy of Science and Technology, 18 Hoang Quoc Viet Road, 10307
Hanoi, Vietnam\\
{\it E-mail address}: ntcuong@math.ac.vn
\vskip0.1cm
\noindent
Department of Mathematics, FPT University, 8 Ton That Thuyet Road, Ha Noi, Vietnam\\
 {\it E-mail address}: quyph@fpt.edu.vn
\vskip0.1cm
\noindent
Institute of Mathematics, Vietnam Academy of Science and Technology, 18 Hoang Quoc Viet Road, 10307
Hanoi, Vietnam\\
{\it E-mail address}: hltruong@math.ac.vn
\end{document}